\documentclass{amsproc}
\usepackage{amsmath,amsthm,amsfonts,graphics,amssymb,mathrsfs,showidx,paralist,underscore}
\usepackage[dvips]{graphicx}
\usepackage[all]{xy}

\usepackage{color}
\definecolor{abelian}{cmyk}{0.50,0,1,.4}
\definecolor{noabelian}{cmyk}{0.94,0.54,0,0}
\definecolor{rojo}{cmyk}{0,1,1,0}
\definecolor{verde}{cmyk}{0.91,0,0.88,0.12}
\xyoption{arc}

\newtheorem{theorem}{Theorem}[section]

\theoremstyle{definition}
\newtheorem{definition}[theorem]{Definition}

\newtheorem{question}[theorem]{Question}

\theoremstyle{remark}

\numberwithin{equation}{section}

\newcommand{\op}{\operatorname}

\newcommand{\ben}{\begin{equation}}
\newcommand{\een}{\end{equation}}
\newcommand{\bena}{\begin{equation*}}
\newcommand{\eena}{\end{equation*}}

\newcommand{\Tot}{\longmapsto}

\def\ZZ{\mathbb{Z}}

\begin{document}

\title{Numerical computations in cobordism categories}

\author{Carlos Segovia}
\address{Mathematisches Institut, Universit\"at Heidelberg, Deutschland}
\curraddr{}
\email{csegovia@mathi.uni-heidelberg.de}
\thanks{}


\date{13.03.2013}


\begin{abstract} The sequence 2,5,15,51,187,... with the form $(2^n+1)(2^{n-1}+1)/3$ has two interpretations in terms of the density of a language with four letters \cite{mor} and the cardinality of the quotient of $\ZZ_2^n\times \ZZ_2^n$ under the action of the special linear group $\op{SL}(2,\ZZ)$. The last interpretation follows the rank of the fundamental group of the $\ZZ_2^n$-cobordism category in dimension 1+1, see \cite{carlos1}. This article presents how to pass from one side to another between these two approaches.
\end{abstract}

\maketitle

\section*{Introduction}
It was a knock out when after writing the sequence 2,5,15,51,187,715,... in the page \cite{oeis}, we found a great variety of different interpretations for it. This sequence was presented to the author for first time in some numerical computations in his PhD dissertation, see \cite{carlos1}\footnote{The group $\ZZ_2$ is the only group with just two elements.}. There exists more that four contrasted approaches, each one inside a totally dissimilar subject and we use the following two paragraphs to give a description of two of them. Finally, in the next section we give a natural connection between them in terms of the binary numeral system and we state an open question.

The density of a language with four letters is defined as follows. Take the number of words of length $n$ made with letters 1,2,3,4 with the property that numbered from left to right each letter satisfies $0<a_i\leq \op{max}_{j\leq i}\{a_j\}+1$. Thus the first letter is always 1, so we can dismiss it. For example, for $n=2$ there are two words $1$ and $2$, for $n=3$ the words are $11$, $12$, $21$, $22$, $23$, while for $n=4$ we have 15 words
\begin{center}\parbox{0cm}{
\begin{tabbing}
  11111\= 11111 \= 11111 \= 11111 \= 11111 \kill
  111 \> 112 \> 121 \> 122 \> 123 \\
  211 \> 212 \> 213 \> 221 \> 222 \\
  223 \> 231 \> 232 \> 233 \> 234\,.
\end{tabbing}}
\end{center}

The second point considers the cardinality of the quotient of $\ZZ_2^n\times \ZZ_2^n$ under the action of the special linear group $\op{SL}(2,\ZZ)$. This group is generated by two matrices
which produce essentially two basic equations $(g,k)\sim (k,-g)$ and $(g,k)\sim (g,k+mg)$. The orbits of this quotient gives a set of generators for the monoid of principal $\ZZ_2^n$-bundles over closed surfaces with two boundary circles up to a homeomorphism identification, see \cite{carlos1}. If you do not understand the last commentary you could go ahead and ignore it. For $n=1$, we get two orbits $(0,0)$ and $(0,1)\sim(1,0)\sim(1,1)$. For $n=2$, we get 5 orbits
\begin{minipage}{8cm}
\begin{enumerate}
\item
$\left(
  \begin{array}{cc}
    0 & 0 \\
    0 & 0 \\
  \end{array}
\right)$
\item
$\left(
  \begin{array}{cc}
    0 & 0 \\
    0 & 1 \\
  \end{array}
\right)\sim \left(
  \begin{array}{cc}
    0 & 0 \\
    1 & 0 \\
  \end{array}
\right)\sim \left(
  \begin{array}{cc}
    0 & 0 \\
    1 & 1 \\
  \end{array}
\right)$
\end{enumerate}
\end{minipage}
\   \
\begin{minipage}{8cm}
\begin{itemize}
\item[(3)]
$\left(
  \begin{array}{cc}
    0 & 1 \\
    0 & 0 \\
  \end{array}
\right)\sim \left(
  \begin{array}{cc}
    1 & 0 \\
    0 & 0 \\
  \end{array}
\right)\sim \left(
  \begin{array}{cc}
    1 & 1 \\
    0 & 0 \\
  \end{array}
\right)$
\item[(4)]
$\left(
  \begin{array}{cc}
    0 & 1 \\
    0 & 1 \\
  \end{array}
\right)\sim \left(
  \begin{array}{cc}
    1 & 0 \\
    1 & 0 \\
  \end{array}
\right)\sim \left(
  \begin{array}{cc}
    1 & 1 \\
    1 & 1 \\
  \end{array}
\right)$
\end{itemize}
\end{minipage}
\begin{itemize}
\item[(5)]
$\left(
  \begin{array}{cc}
    0 & 1 \\
    1 & 1 \\
  \end{array}
\right)\sim
\left(
  \begin{array}{cc}
    1 & 1 \\
    1 & 0 \\
  \end{array}
\right)
\sim
\left(
  \begin{array}{cc}
    1 & 0 \\
    1 & 1 \\
  \end{array}
\right)\sim
\left(
  \begin{array}{cc}
    1 & 1 \\
    0 & 1 \\
  \end{array}
\right)\sim
\left(
  \begin{array}{cc}
    0 & 1 \\
    1 & 0 \\
  \end{array}
\right)\,.$
\end{itemize}
Note that the easiest way to define the sequence $2,5,15,51,187,715,...$ is as the density of a language with four letters.

\section{Main constructions}
Denote the two approaches of the last section as follows:
\begin{enumerate}
\item\label{e1} the density of a language with four letters, and
\item\label{e2} the cardinality of the quotient of $\ZZ_2^n\times \ZZ_2^n$ under the action of the special linear group $\op{SL}(2,\ZZ)$.
\end{enumerate}
Now we define an application from \eqref{e1} to \eqref{e2}. This is defined for a word $a_1a_2...a_n$ by the binary representation for $a_i=2,3\Tot 10,11$; while for $a_i=1,4$ we should take the assignations $00$, $01$ respectively. As an illustration for $n=2$, we have the assignations
\bena
\begin{array}{ccccc}
  \begin{array}{c}
    1 \\
    1
  \end{array}
  \longmapsto
  \left(
       \begin{array}{cc}
         0 & 0 \\
         0 & 0
       \end{array}
     \right)
   &
   \begin{array}{c}
    1 \\
    2
  \end{array}
  \longmapsto
  \left(
       \begin{array}{cc}
         0 & 0 \\
         1 & 0
       \end{array}
     \right)
   &
   \begin{array}{c}
    2 \\
    1
  \end{array}
  \longmapsto
  \left(
       \begin{array}{cc}
         1 & 0 \\
         0 & 0
       \end{array}
     \right)
   &
  \begin{array}{c}
    2 \\
    2
  \end{array}
  \longmapsto
  \left(
       \begin{array}{cc}
         1 & 0 \\
         1 & 0
       \end{array}
     \right)
   &
   \begin{array}{c}
    2 \\
    3
  \end{array}
  \longmapsto
  \left(
       \begin{array}{cc}
         1 & 0 \\
         1 & 1
       \end{array}
     \right)\,.

\end{array}
\eena
When we have words with a letter with value 4 we forget the first two zeros as follows
\bena
\begin{array}{c}
    2 \\
    3 \\
    4
  \end{array}
  \longmapsto
         \begin{array}{cccc}
         & & 1 & 0 \\
         & & 1 & 1 \\
         0& 1 & 0 & 0
       \end{array}
  \longmapsto
     \left(
       \begin{array}{cccc}
          1 & 0 \\
          1 & 1 \\
          0 & 1
       \end{array}
     \right)\,.
\eena
This application is surjective since the identification in \eqref{e2} is just column operations, which is used to transform every element by one which is the image of a word from \eqref{e1}. In \cite{mor} it is proved that the density of a language with four letters has the value $(2^n+1)(2^{n-1}+1)/3$. We prove in theorem \ref{a1} that for the second approach we have the same value. Thus this application between these two approaches should be a bijection.

\begin{definition}
For an abelian group $G$ denote $r(G)$ the cardinality of the quotient of $G\times G$ up to the identification generated by the equations $(g,k)\sim (k,-g)$ and $(g,k)\sim(g,k+mg)$.
\end{definition}

\begin{theorem}\label{a1} For the group $\ZZ_p^n$, with $p$ a prime number, we have the
identity \bena
r(\ZZ_p^n)=\frac{p^{2n-1}+p^{n+1}-p^{n-1}+p^2-p-1}{p^2-1}\,.\eena
\end{theorem}
\begin{proof}
For $r_p^n:=r(\ZZ_p^n)$, let $F(n)$ be the number $r_p^{n+1}-r_p^n$. We will prove that
\ben\label{from1} F(n)=p^{n-1}(p^n+p-1)\,.\een Since
$r_p^{n}=(r_p^{n}-r_p^{n-1})+(r_p^{n-1}-r_p^{n-2})+...+(r_p^3-r_p^2)+(r_p^2-r_p^1)+r_p^1$,
where $r_p^1=2$. Thus
$r_p^{n}=p^{n-2}(p^{n-1}+p-1)+p^{n-3}(p^{n-2}+p-1)+...+p(p^2+p-1)+(p+p-1)+2$
and as a consequence we have the following equations
\begin{align*}
r_p^{n}&=\sum_{i=0}^{n-2}p^{2i+1}+(p-1)\sum_{i=0}^{n-2}p^i+2\\
&=
p\frac{(p^2)^{n-1}-1}{p^2-1}+(p-1)\frac{p^{n-1}-1}{p-1}+2\\
&=\frac{p^{2n-1}-p+(p^{n-1}-1)(p^2-1)}{p^2-1}\\
&=\frac{p^{2n-1}+p^{n+1}-p^{n-1}+p^2-p-1}{p^2-1}\,.
\end{align*}

We will prove the formula \eqref{from1} by induction, where we use the following
\ben\label{a2} F(n)=pF(n-1)+p^{2n-2}(p-1)\,.\een
We end with the proof of the formula \eqref{a2}. By definition $F(n)$ consists of elements in
$\ZZ_p^{n+1}\times\ZZ_p^{n+1}$ such that the last row
is different from the zero row. There are three cases to consider:
\begin{enumerate}
\item[1)] the representatives of the classes have zeros in the
$n$-coordinate, i.e. of the
form \bena \left(%
\begin{array}{cc}
  \vdots& \vdots\\
   0 &  0\\
   i & j
\end{array}%
\right)\,,\eena with $i,j\neq 0$ at least one, and for this case the number
of classes is $F(n-1)$;

\item[2)] the representatives of the classes that have zeros in the
second column for the last two rows, i.e. of the
form
\bena\label{a123} \left(%
\begin{array}{cc}
  \vdots & \vdots \\
   i & 0 \\
   j & 0
\end{array}%
\right)\,,\eena with $i\neq 0$ and $j\neq 0$. For these elements the stabilizer group is the same, before and after erasing the last column, so we have
$(p-1)F(n-1)$ classes, where we multiply by $p-1$ since we can not take the zero value for $j$; and

 \item[3)] the last case is composed by classes with a representative of the form
\ben\label{a1234} \left(%
\begin{array}{cc}
  \vdots & \vdots\\
  i & 0 \\
  0 & j
\end{array}%
\right)\,,\een
with $i\neq 0$ and $j\neq 0$. Every stabilizer of an
element of the form \eqref{a1234} has to be the identity, then the
classes have as cardinality the order of $\op{SL}(2,\ZZ_p)$, which is
$p(p^2-1)$. Thus the number of classes is given by the product of $p^{n-1}p^{n-1}$ (given by the first $(n-1)$-elements of the two columns) product with $\left|\op{GL}(2,\ZZ_p)\right|$(given by the part $\small{\left(
                                                                         \begin{array}{cc}
                                                                           i & 0 \\
                                                                           0 & j \\
                                                                         \end{array}
                                                                       \right)}\in\op{GL}(2,\ZZ_p)
$ in \eqref{a1234}) and divided by $\left|\op{SL}(2,\ZZ_p)\right|$. Therefore, the number of classes is
$p^{2n-2}(p-1)$. Finally, the sum of the numbers associated to these three
cases, gives equation \eqref{a2} and we have the proof of the theorem.
\end{enumerate}
\end{proof}

\begin{question}
What is the analog of the approach \eqref{e1} for every prime number?
\end{question}
\bibliographystyle{amsalpha}
\bibliography{biblio}

\providecommand{\bysame}{\leavevmode\hbox to3em{\hrulefill}\thinspace}
\providecommand{\MR}{\relax\ifhmode\unskip\space\fi MR }
\providecommand{\MRhref}[2]{%
  \href{http://www.ams.org/mathscinet-getitem?mr=#1}{#2}
}
\providecommand{\href}[2]{#2}
\begin{thebibliography}{MR05}

\bibitem[MR05]{mor}
Nelma Moreira and Rog\'erio Reis, \emph{On the density of languages
  representing finite set partitions}, Journal of Integer Sequences \textbf{8}
  (2005), 1--11.

\bibitem[oei]{oeis}
\emph{The on-line encyclopedia of integer sequences}.

\bibitem[Seg12]{carlos1}
Carlos Segovia, \emph{The classifying space of the 1+1 dimesional $g$-cobordism
  category}, http://arxiv.org/abs/1211.2144, Nov 2012.

\end{thebibliography}
\end{document}